\documentclass[12pt]{amsart}

\input{xy}
\xyoption{all}
\usepackage{amssymb}
\usepackage{graphicx}
\usepackage{color,ulem}
\usepackage{verbatim}
\usepackage[bookmarks,colorlinks,breaklinks]{hyperref}  % PDF hyperlinks, with coloured links
\hypersetup{linkcolor=blue,citecolor=blue,filecolor=blue,urlcolor=blue} % all blue links

\newtheorem{theorem}{Theorem} [section]
\newtheorem{prop}[theorem]{Proposition}

\newtheorem{question}[theorem]{Question}

\theoremstyle{definition}
\newtheorem{example}[theorem]{Example}
\newtheorem{remark}[theorem]{Remark}

\numberwithin{equation}{section}
\numberwithin{figure}{section}

\newcommand\C{{\mathbb C}}
\newcommand\Chat { {\hat{\C}} } 

\renewcommand\P{{\mathbb P}}

\newcommand\R{{\mathbb R}}

\newcommand\Q{{\mathbb Q}}
\newcommand\N{{\mathbb N}}

\newcommand\cM{\mathcal{M}}
\newcommand\cC{\mathcal{C}}

\renewcommand\phi{\varphi}
\newcommand\Gal{\operatorname{Gal}}

\newcommand\Qbar{\overline{\mathbb{Q}}}

\newcommand\Kbar{\overline{K}}

\newcommand\del{\partial} 
 
\newcommand\iso{\simeq} %isomorphism -~

\newcommand\codim {\operatorname{codim}} % codimension
\definecolor{myblue}{rgb}{0.6, 0.9, 1}
\definecolor{mygreen}{rgb}{0,0,1}
\definecolor{purple}{rgb}{0.6,0.2,1}
\definecolor{orange}{rgb}{0.8,0,0.2}

\begin{document}

\title[Geometry of PCF parameters]{Geometry of PCF parameters in spaces of quadratic polynomials}

%\title[Optimal general uniformity]{Optimal uniform bounds for PCF points on general curves in spaces of quadratic polynomials}

\author{Laura De Marco and Niki Myrto Mavraki}
\email{demarco@math.harvard.edu}
\email{myrto.mavraki@utoronto.ca}
%\email{yehexi@zju.edu.cn}

\date{\today}

\begin{abstract}
We study algebraic relations among postcritically finite (PCF) parameters in the family $f_c(z) = z^2 + c$.  In \cite{GKNY}, it was proved that an algebraic curve in $\C^2$ contains infinitely many PCF pairs $(c_1, c_2)$ if and only if the curve is special (i.e., the curve is a vertical or horizontal line through a PCF parameter, or the curve is the diagonal).  Here we extend this result to subvarieties of arbitrary dimension in $\C^n$ for any $n\geq 2$.  Consequently, we obtain uniform bounds on the number of PCF pairs on non-special curves in $\C^2$ and the number of PCF parameters in real algebraic curves in $\C$, depending only on the degree of the curve.  We also compute the optimal bound for the general curve of degree $d$.  For $d=1$, we prove that there are only finitely many non-special lines in $\C^2$ containing more than two PCF pairs, and similarly, that there are only finitely many (real) lines in $\C = \R^2$ containing more than two PCF parameters.
\end{abstract}

% \subjclass{(2010 Classification) Primary 37F45, Secondary 37P30, 11G05}
%\keywords{}

%\thanks{The research was supported by the National Science Foundation.}

\maketitle

\thispagestyle{empty}

%%%%%%
%%%%%%
\bigskip
\section{Introduction}

For each $c \in \C$, let $f_c(z) = z^2 + c$.  Recall that the polynomial $f_c$ is postcritically finite (PCF) if the critical point at $z_0 = 0$ has a finite forward orbit.  In this article, we study algebraic relations among the PCF parameters $c \in \C$.

\begin{figure} [h]
\includegraphics[width=3.5in]{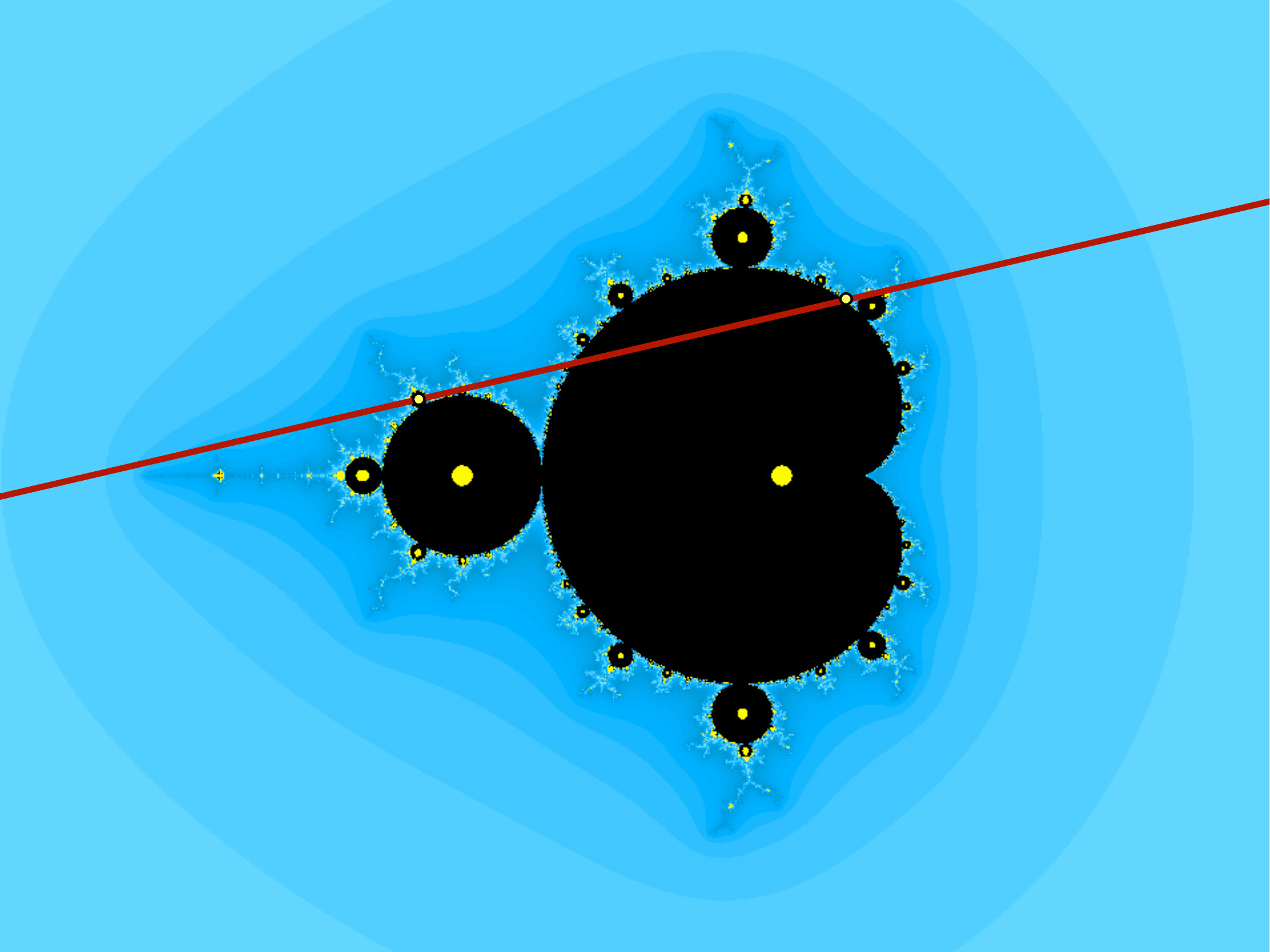}
\caption{ \small The Mandelbrot set with PCF parameters marked in yellow. There are only finitely many real lines in $\C$ containing more than two PCF parameters; see Theorem \ref{real}.}
\label{Mandelbrot line}
\end{figure}

Our starting point is the following theorem of Ghioca, Krieger, Nguyen, and Ye \cite{GKNY}.  Generalizations to algebraic curves in other polynomial families were obtained by Favre and Gauthier \cite{Favre:Gauthier:book}.

\begin{theorem} \cite{GKNY} \label{M^2}
Let $C$ be an irreducible complex algebraic curve in $\C^2$. Then $C$ contains infinitely many PCF pairs $(c_1, c_2)$ if and only if $C$ is either
\begin{enumerate}
\item a vertical line $\{x = c_1\}$ for a PCF $f_{c_1}$; or
\item a horizontal line $\{y = c_2\}$ for a PCF $f_{c_2}$; or
\item the diagonal $\{x=y\}$
\end{enumerate}
in coordinates $(x,y)$ on $\C^2$.
\end{theorem}

\noindent
Note that a real algebraic curve in $\R^2$ passing through a PCF parameter $c_0$ in the Mandelbrot set (identifying $\R^2$ with $\C$) gives rise to a complex algebraic curve in $\C^2$ passing through the PCF pair $(c_0, \bar{c}_0)$.  So the above result also controls PCF points on real curves in $\C$.  See Figure \ref{Mandelbrot line} and Section \ref{real lines}.

Theorem \ref{M^2} was motivated by analogies between the PCF maps in the space of quadratic polynomials and the elliptic curves with complex multiplication (CM) in the space of $j$-invariants; see for example \cite[Ch. 6]{Silverman:moduli} \cite[Conj. 3.11]{Jones:survey}.  It was known that the only algebraic curves in $\C^2$ with infinitely many CM pairs are the modular curves (together with the infinite collection of vertical or horizontal lines through a CM point) \cite{Andre:finitude}, \cite{Edixhoven:special2}. 

Our first result is an extension of Theorem \ref{M^2} to arbitrary dimensions, exactly analogously to the classification of special subvarieties in the CM case \cite{Pila:AO}, \cite{Edixhoven:special}:

\begin{theorem} \label{M^n}
Let $n\ge 2$. Let $X$ be an irreducible complex algebraic subvariety in $\C^n$.  There is a Zariski dense set of special points in $X$ if and only if $X$ is special. 
\end{theorem}

By definition, a parameter $c \in \C$ is {\bf special} if $f_c$ is PCF.  For any positive integer $n$, a point $(c_1, \ldots, c_n)\in\C^n$ is \textbf{special} if $c_i$ is special for all $i = 1, \ldots, n$.  We say that an irreducible curve $C\subset \mathbb{C}^2$ is \textbf{special} if it is one of the three types listed in Theorem \ref{M^2}.  The {\bf special subvarieties} of $\C^n$ are the preimages of special curves from projections to $\C^2$, and their intersections.  More precisely, an irreducible subvariety $Z$ of $\C^n$ is special if and only if there exist a partition $S_0 \cup \cdots \cup S_r$ of $\{1, \ldots, n\}$, where $r\geq 0$ and $S_k \not= \emptyset$ for each $k>0$, and a collection of PCF parameters $c_i \in \C$ for $i \in S_0$ 
%, and integers $N_{k,j} \in \N$ for each $j \not= i_k := \min S_k$ with $k >0$, 
so that 
	$$Z = \left( \bigcap_{i \in S_0} \{x_i = c_i\}\right) \cap \left( \bigcap_{k=1}^r \bigcap_{j  \in S_k} \{x_j = x_{i_k}\} \right)$$
where $(x_1, \ldots, x_n)$ are the coordinates of $\C^n$ and $i_k := \min S_k$ for each $k >0$.  Note that the dimension of $Z$ is equal to $r$.

Although Theorem \ref{M^n} is worded the same as statements about modular curves, the proof methods are (necessarily) very different.  As in the proof of Theorem \ref{M^2}, it is important that the PCF parameters are a set of algebraic numbers with bounded Weil height, which is not the case for singular moduli, and in fact of height 0 for a dynamically-defined height on $\P^1(\Qbar)$.  This allows the use of certain arithmetic equidistribution theorems for points of small height; we rely on the recent equidistribution result of Yuan-Zhang \cite{Yuan:Zhang:quasiprojective} (though we could have used the older result of \cite{Yuan:equidistribution} as we explain in Remark \ref{other eq}).  Focusing then on an archimedean place, and via the slicing of positive currents, we reduce the proof of Theorem \ref{M^n} to Luo's theorem on the inhomogeneity of the Mandelbrot set \cite{Luo:inhomogeneity}.

As an application of Theorem \ref{M^n}, we obtain uniform versions of Theorem \ref{M^2}, in the spirit of Scanlon's automatic uniformity \cite{Scanlon:automatic} (though we give a direct proof, not relying on \cite[Theorem 2.4]{Scanlon:automatic}).
%and inspired by uniformity results in the setting of families of abelian varieties [REFERENCES, including Bombieri-Masser-Zannier, Zilber?].

\begin{theorem}\label{uniformChow}
Fix $d\in\N$. There is a constant $M(d) < \infty$ such that 
$$\# \{\mbox{special points in } C \}\le M(d),$$
for all complex algebraic curves $C \subset \C^2$ of degree $d$ without special components.
\end{theorem}

It is natural to ask how many special points can lie on a non-special curve in $\C^2$.  We obtain an explicit bound for the general curve of degree $d$:  

\begin{theorem}\label{optimal}
Fix $d\in\N$, and let $X_d$ denote the Chow variety of all plane curves with degree $\leq d$.
There exists a Zariski-closed strict subvariety $V_d\subset X_d$ such that 
	$$\# \{\mbox{special points in } C \} \le \frac{d(d+3)}{2},$$
for all curves $C\in X_d\setminus V_d$. 
\end{theorem}

The upper bound in Theorem \ref{optimal} is optimal: 

\begin{theorem} \label{sharp}
There is a Zariski-dense subset $S_d \subset X_d$ such that 
$$\# \{\mbox{special points in } C \} = \frac{d(d+3)}{2},$$
for all curves $C\in S_d$. 
\end{theorem}

Note that $d(d+3)/2$ is the dimension of the space $X_d$, and this is no accident.  It is well known that there exists a curve of degree $d$ through any collection of $N_d = d(d+3)/2$ points in $\C^2$.  Choosing those points to be special, we can build a Zariski-dense collection of curves in $X_d$ containing at least $N_d$ special points.  The upper bound of Theorem \ref{optimal} is obtained by showing there are no unexpected symmetries among general special-point configurations, as a consequence of Theorem \ref{M^n} and the explicit description of the special subvarieties.

Our proof does not give a complete description of the exceptional variety $V_d$ in Theorem \ref{optimal}, though the methods can be used to classify its positive-dimensional components.  For example, in the case of $d=1$, we show:

\begin{theorem}\label{lines}
All but finitely many non-special lines in $\C^2$ contain at most $2$ special points. 
\end{theorem}

\noindent
In other words, the subvariety $V_1$ of Theorem \ref{optimal} can be taken to be the union of the 1-parameter families in $X_1$ of horizontal and vertical lines, together with a finite set of points in $X_1$.   A more detailed result about lines in $\C^2$ is stated as Proposition \ref{prop M lines}.  The analogue of Theorem \ref{lines} in the setting of CM points in $\C^2$ was proved in \cite{Bilu:Luca:Masser}.

Note that the finite set of non-special lines in $\C^2$ containing at least 3 special points is not empty.  For example, the line $\{y  = -x\}$ passes through $(0,0)$, $(i, -i)$, and $(-i, i)$; the line $\{y = i \, x\}$ passes through $(0,0)$, $(-1, -i)$, and $(i, -1)$; and $\{y = -i \, x\}$ passes through $(0,0)$, $(-1, i)$, and $(-i, -1)$.

\begin{question}
How many non-special lines in $\C^2$ pass through at least 3 distinct special points, and what is the optimal value of $M(1)$ in Theorem \ref{uniformChow}? 
\end{question}

As mentioned after Theorem \ref{M^2}, a real algebraic curve in $\R^2 = \C$ passing through given parameter $c$ in the Mandelbrot set will give rise to a complex algebraic curve in $\C^2$ passing through point $(c, \bar{c})$ (see Section \ref{real lines}).  Moreover, the subset of such curves is Zariski dense in $X_d$ for each degree $d \geq 1$.  Theorems \ref{uniformChow} and \ref{optimal} therefore apply to bound PCF parameters on real algebraic curves of a given degree in $\R^2 = \C$.  For example, we have:

\begin{theorem} \label{real}
There is a uniform bound on the number of PCF parameters on any real algebraic curve in $\R^2 = \C$ depending only on the degree of the curve (as long as the curve does not contain the real axis).  Moreover, there are only finitely many real lines in $\C$ that contain more than two PCF parameters.
\end{theorem} 

\noindent
Note that the finite set of real lines in $\C$ containing more than two PCF parameters is not empty:  the real axis contains infinitely many and the imaginary axis contains at least 3 (at $c = 0$ and $c = \pm i$).  

\begin{remark}
Finiteness results analogous to Theorem \ref{lines}, upon replacing ``lines" with ``curves of degree $d$" and the bound of 2 with $d(d+3)/2$, do not hold for $d>1$.  For example, for algebraic curves of degree $d=2$, we know there is a conic through any 5 given points in $\C^2$, and 5 is the optimal bound on special points in general conics (by Theorems \ref{optimal} and \ref{sharp}), but there is a Zariski-dense set of curves in the 3-dimensional space of conics 
	$$x^2 + y^2 + A \, x \, y + B\, (x+y) + C = 0$$
in $\C^2$ containing at least 6 special points.  Indeed, 3 given special points in $\C^2$ will (generally) determine the coefficients $A,B, C$, and the $(x,y) \mapsto (y,x)$ symmetry of the curve will (generally) guarantee an additional 3 special points.  For real conics in $\R^2 = \C$, one can do the same with symmetry under complex conjugation; see Figure \ref{ellipse}. 
\end{remark}

\begin{figure} [h]
\includegraphics[width=2.75in]{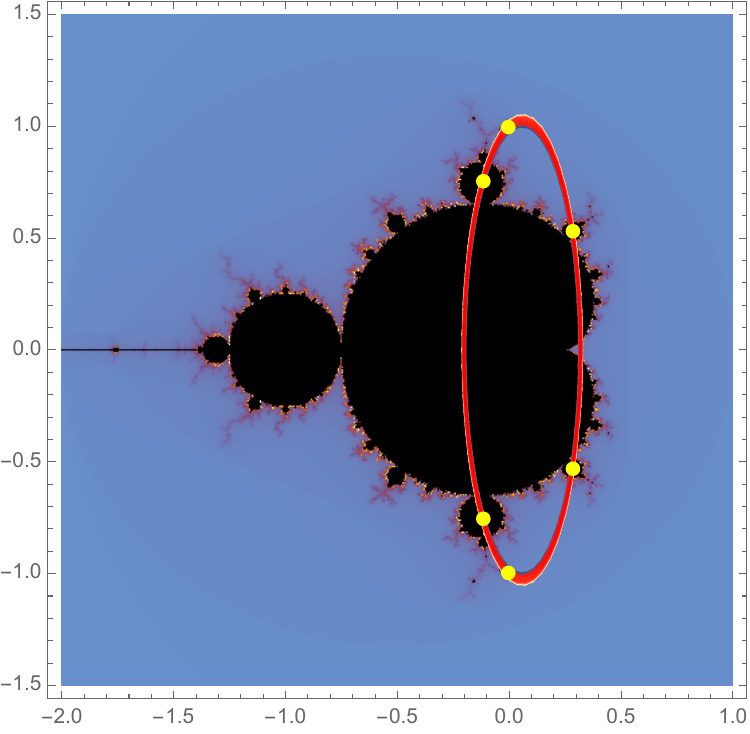}
\caption{ \small A general real conic in $\R^2$ will contain no more than 5 PCF parameters, but there are infinitely many symmetric conics with at least 6 special points.}
\label{ellipse}
\end{figure}

\bigskip\noindent{\bf Outline.}  In Section \ref{n special}, we prove Theorem \ref{M^n}.  Section \ref{maximal} provides a brief review of the Chow variety $X_d$ and basic results on families of curves passing through points.  Section \ref{upper} contains the proofs of Theorems \ref{uniformChow}, \ref{optimal}, \ref{sharp}, and \ref{lines}.  Finally, in Section \ref{real lines}, we look at real algebraic curves passing through PCF parameters in the Mandelbrot set and prove Theorem \ref{real}.

\bigskip\noindent{\bf Acknowledgements.}  Special thanks go to Hexi Ye, as this article grew out of conversations during his visit to Harvard in early 2023.  We are also grateful to Gabriel Dill, Holly Krieger, and Jacob Tsimerman for their questions and suggestions that led to Theorems \ref{uniformChow} and \ref{optimal}, and we thank Thomas Gauthier for helpful comments on an early version of this article. The figures were generated with Dynamics Explorer (developed by Brian and Suzanne Boyd) and Wolfram Mathematica.  Our research was supported by the National Science Foundation.  %The authors acknowledge the support of NSF grants DMS-2200981 and DMS-2246630.  

%%%%%%%%%%

\bigskip
\section{Proof of Theorem \ref{M^n}}
\label{n special}

In this section we prove Theorem \ref{M^n}. In fact we prove a stronger result, showing that our classification theorem remains true if we treat small points in addition to the special points. Our notion of size is given by a height function 
\begin{equation} \label{h crit}
h_{\mathrm{crit}}(c_1,\ldots,c_n) := \sum_{i=1}^n\hat{h}_{f_{c_i}}(0) \geq 0
\end{equation}
for $(c_1, \ldots, c_n) \in \Qbar^n$.  Here $\hat{h}_{f_{c_i}}$ is the canonical height associated to the quadratic polynomial $z^2+c_i$, introduced by Call and Silverman \cite{Call:Silverman}. 
We say that a sequence $\{x_k\}\subset \Qbar^n$ is  \textbf{small} if $h_{\mathrm{crit}}(x_k)\to 0$ as $k \to \infty$.  Notice that our special points of $\C^n$ are precisely the zeros of $h_{\mathrm{crit}}$.  

Let $Y\subset \C^n$ be a variety.  A sequence $\{x_k\}\subset Y$ is called \textbf{generic} if no subsequence lies in a proper subvariety of $Y$.

\begin{theorem}\label{eM^n}
Let $n\ge 2$. Let $X$ be an irreducible algebraic subvariety in $\C^n$ defined over $\Qbar$. Then, $X$ contains a generic sequence of small points  if and only if $X$ is special. 
\end{theorem}

The idea of considering points that are small with respect to some height function originates in Bogomolov's work \cite{Bogomolov:original}; in a dynamical context, see for example \cite[Conjecture 2.3]{GHT:rational} or \cite{Zhang:distributions}.

\begin{remark}  \label{qbar}
In Theorem \ref{eM^n} we have assumed that $X$ is defined over $\Qbar$, which is not the case in Theorem \ref{M^n}. However, our special points are defined over $\Qbar$ so that a subvariety that contains a Zariski dense set of special points is automatically defined over $\Qbar$. Therefore, Theorem \ref{M^n} follows from Theorem \ref{eM^n}. Note here that the structure of the special subvarieties ensures that they contain a Zariski dense set of special points. 
\end{remark}

\begin{remark}  Assuming $X$ is a curve in $\C^2$, the conclusion of Theorem \ref{eM^n} is not contained in \cite{GKNY} but follows immediately from the proof of Theorem B in \cite{Favre:Gauthier:book}. % In \cite{GKNY} the authors relied on the existence of special points on the curve in their proof of \cite[Proposition 5.6]{GKNY}.
\end{remark}

\subsection{Arithmetic equidistribution}
The first key ingredient in our proof is the following equidistribution theorem.  Let $\cM$ denote the Mandelbrot set in $\C$.  Let $\mu_{\cM}$ denote the bifurcation measure on $\cM$.  As computed in \cite[\S6]{D:current}, $\mu_{\cM}$ is proportional to the harmonic measure supported on $\del \cM$ for the domain $\Chat\setminus \cM$, relative to the point at $\infty$.  The support of $\mu_{\cM}$ is equal to the boundary of $\cM$; it has continuous potentials and total mass equal to $\frac12$.

\begin{theorem}\label{equidistribution}
Let $n\ge 2$ and $H\subset \C^{n}$ be an irreducible hypersurface defined over a number field $K$. Assume that the projection $p_j:H\to \C^{n-1}$ omitting the $j$-th coordinate is dominant. 
Then, for any generic sequence $\{x_k\}\subset H(\Kbar)$ of small points, their $\Gal(\overline{K}/K)$-orbits equidistribute to the probability measure  
$$ \mu_j:= c \, (\pi_1|_{H})^{*}(\mu_{\cM})\wedge\cdots\wedge(\pi_{j-1}|_{H})^{*}(\mu_{\cM})\wedge(\pi_{j+1}|_{H})^{*}(\mu_{\cM})\wedge\cdots\wedge(\pi_n|_{H})^{*}(\mu_{\cM})$$
on $H(\C)$, where $\pi_i: \C^n\to \C$ is the projection to the $i$-th coordinate, $(\pi_i|_H)^*\mu_{\cM}$ is the pullback as a (1,1)-current, and $c>0$ is a positive constant. That is, for any continuous function $\phi$ on $H$ with compact support in the smooth part of $H$, we have  
	$$\frac{1}{\# \, \Gal(\Kbar/K) \cdot x_k} \sum_{y \in \Gal(\Kbar/K) \cdot x_k} \phi(y) \longrightarrow \int \phi \, d\mu_j$$
as $k\to\infty$. 
\end{theorem}

To prove Theorem \ref{equidistribution} we rely on the recent theory of Yuan and Zhang on adelic line bundles developed in \cite{Yuan:Zhang:quasiprojective}. 
We let $f: \mathbb{A}^1\times \mathbb{P}^1\to \mathbb{A}^1\times \mathbb{P}^1$ be the algebraic family of unicritical quadratic polynomials 
	$$f(t,z)=(t,z^2+t),$$
defined over $\Q$. 
Let $L$ be the line bundle on $\mathbb{A}^1\times \mathbb{P}^1$, isomorphic to $\mathcal{O}(1)$ on fibers $\mathbb{P}^1$ and such that $f^{*}L=2L$. We denote by $\overline{L}_f$ the $f$-invariant extension of $L$ as defined in \cite[Theorem 6.1.1]{Yuan:Zhang:quasiprojective}. 
Let $i: \mathbb{A}^1\to \mathbb{A}^1\times  \mathbb{P}^1$ be defined by $i(t)=(t,0)$ and 
define 
\begin{align*}
\overline{L}_{\mathrm{crit}}:=i^{*}\overline{L}_f.
\end{align*}
This is an adelic line bundle on $\mathbb{A}^1$ as in \cite[\S6.2.1]{Yuan:Zhang:quasiprojective}. 
Furthermore, 
by \cite[Lemma 6.2.1]{Yuan:Zhang:quasiprojective}, the height associated to $\overline{L}_{\mathrm{crit}}$ is given by 
\begin{align}\label{bif height}
h_{\overline{L}_{\mathrm{crit}}}(t)=\hat{h}_{f_t}(0)=h_{\mathrm{crit}}(t),
\end{align}
for each $t\in \mathbb{A}^1(\Qbar)$, where $h_{\mathrm{crit}}$ is the height defined in \eqref{h crit} with $n=1$.  By construction, we have 
\begin{equation} \label{arch c1}
	c_1(\overline{L}_{\mathrm{crit}})=\mu_{\cM}
\end{equation}
at the archimedean place of $\Q$.

%$\overline{L}_{f}$ as in \cite[\S 6.2.2]{YZv5} It has height $h_{\mathrm{bif}}$ and curvature form $\mu_{\cM}$. 

\begin{remark} \label{other eq}
To prove Theorem \ref{equidistribution} we use the recent equidistribution theory in quasiprojective varieties developed in \cite{Yuan:Zhang:quasiprojective}. However, it is known that $\overline{L}_{\mathrm{crit}}$ extends to define an adelic metrized line bundle on the projective line $\mathbb{P}^1$; see, for example, \cite[\S 6.5]{FRL:quantitative}.   Arguments similar to the ones in \cite{Mavraki:Schmidt:Wilms} would allow us to use the equidistribution result established by Yuan \cite{Yuan:equidistribution} instead, applied to a projective compactification of the $H$ in Theorem \ref{equidistribution} and a modification of the metrized line bundles $\overline{M}_j$ we define in \eqref{Mj} below, but the results from \cite{Yuan:Zhang:quasiprojective} considerably simplify the exposition here.
\end{remark}

\begin{proof}[Proof of Theorem \ref{equidistribution}]
Fix $j$ as in the statement of the theorem.  As $\mu_{\cM}$ has continuous potentials, we deduce that $\mu_j$ does not put mass on the singular locus $H^{\mathrm{sing}}$ of $H$, so we may replace $H$ with $H\setminus H^{\mathrm{sing}}$ and assume that $H$ is smooth.
We define a metrized line bundle on the smooth, quasiprojective hypersurface $H$ by
\begin{align} \label{Mj}
\overline{M}_j=\displaystyle\otimes_{i\neq j}(\pi_i|_{H})^{*}(\overline{L}_{\mathrm{crit}}). 
\end{align}
This defines an adelic line bundle on $H$, so that $\overline{M}_j\in \widehat{\mathrm{Pic}}(H)_{\mathbb{Q}}$ in the notation of \cite{Yuan:Zhang:quasiprojective}.
By \cite[Theorem 6.1.1]{Yuan:Zhang:quasiprojective} we know that $\overline{L}_f$ is nef in the sense of \cite{Yuan:Zhang:quasiprojective} and by the functoriality of nefness we also have that $\overline{M}_j$ is nef; see \cite[page 8]{Yuan:Zhang:quasiprojective}.  
In what follows we work with the standard absolute on $\mathbb{C}$.
As in \cite[Lemma 6.3.7]{Yuan:Zhang:quasiprojective}, we know from \eqref{arch c1} that the curvature form associated to $\overline{L}_{\mathrm{crit}}$ (at the archimedean place of $\Q$) is equal to $\mu_{\cM}$.  Since $c_1(\overline{L}_{\mathrm{crit}})^{\wedge 2}\equiv 0$, we thus see that 
%$$c_1(\overline{M}_j)^{\wedge (n-1)}= (n-1)! \, (\pi_1|_{H})^{*}(\mu_{\cM})\wedge\cdots\wedge(\pi_{j-1}|_{H})^{*}(\mu_{\cM})\wedge(\pi_{j+1}|_{H})^{*}(\mu_{\cM})\wedge\cdots\wedge(\pi_n|_{H})^{*}(\mu_{\cM}),$$
$$c_1(\overline{M}_j)^{\wedge (n-1)}= (n-1)! \, (\pi_1|_{H})^{*}(\mu_{\cM})\wedge\cdots\wedge\widehat{(\pi_{j}|_{H})^{*}(\mu_{\cM})}\wedge\cdots\wedge(\pi_n|_{H})^{*}(\mu_{\cM}),$$
where the $\widehat{\,\cdot\,}$ means the $j$th term is omitted and the pullbacks are defined in the sense of currents. 
Our assumption that the projection $p_j$ is dominant ensures that this measure is non-trivial. By \cite[Lemma 5.4.4]{Yuan:Zhang:quasiprojective}, we infer that $\overline{M}_j$ is non-degenerate as defined in \cite[\S 6.2.2]{Yuan:Zhang:quasiprojective}.
In other words, the adelic line bundle $\overline{M}_j$ satisfies all the assumptions of \cite[Theorem 5.4.3]{Yuan:Zhang:quasiprojective}. Thus, if $\{y_k\}\subset H(\Qbar)$ is a generic sequence with $h_{\overline{M}_j}(y_k)\to h_{\overline{M}_j}(H)$, then its Galois conjugates equidistribute with respect to the probability measure associated to $c_1(\overline{M}_j)^{\wedge (n-1)}$. 

Now let $\{x_k\}\subset H(\Qbar)$ be a generic sequence of small points, as in the statement of the theorem. 
Note that $$h_{\overline{M}_j}(x)=\displaystyle\sum_{i\neq j}h_{\mathrm{crit}}(\pi_i(x)),$$
for $x \in H(\Qbar)$ so that by \eqref{bif height} we have 
\begin{align}\label{small sequence}
\lim_{k\to\infty} h_{\overline{M}_j}(x_k)=0.
\end{align}
Therefore by the number field case of the fundamental inequality \cite[Theorem 5.3.3]{Yuan:Zhang:quasiprojective}, we have that 
\begin{align}
 h_{\overline{M}_j}(H)\le 0. 
\end{align}
Note that here we have used the fact that $\overline{M}_j$ is nef and non-degenerate. By the nefness of $\overline{M}_j$ we also have that $ h_{\overline{M}_j}(H)\ge 0$ by \cite[Proposition 4.1.1]{Yuan:Zhang:quasiprojective}. Thus, 
\begin{align*}
 h_{\overline{M}_j}(H)=0. 
\end{align*}
Therefore, the result follows by the equidistribution theorem \cite[Theorem 5.4.3]{Yuan:Zhang:quasiprojective}.
\end{proof}

\subsection{Inhomogeneity of $\cM$}
To deduce Theorem \ref{eM^n}, we will combine Theorem \ref{equidistribution} with the following result of Luo that the Mandelbrot set has no local symmetries. 

\begin{theorem} \cite{Luo:inhomogeneity} \label{Luo}
Let $U$ be an open set in $\C$ with $U \cap \del \cM \not= \emptyset$.  Suppose $\phi: U \to V$ is a conformal isomorphism such that $\phi(U \cap \del \cM) = V \cap \del \cM$.  Then $\phi$ is the identity.
\end{theorem}

We begin by proving Theorem \ref{eM^n} for a certain class of hypersurfaces.

\begin{prop}\label{dominant hyper}
Let $n\ge 2$. Assume that $H\subset \C^n$ is an irreducible hypersurface, defined over $\Qbar$, which projects dominantly on each collection of $n-1$ coordinates and which contains a generic sequence of small points. Then $n=2$ and $H\subset \C^2$ is the diagonal line.
\end{prop}
\begin{proof}
Let $H$ be a hypersurface as in the statement defined over a number field $K$. In particular $H$ contains a generic small sequence. 
Since $H$ projects dominantly on each collection of $n-1$ coordinates, we may apply Theorem \ref{equidistribution} to deduce that the $\mathrm{Gal}(\overline{K}/K)$-orbits of our sequence equidistribute with respect to $\mu_j$ for all $j$ (for the $\mu_j$ in the statement of Theorem \ref{equidistribution}). 
In particular 
\begin{align}\label{tobesliced}
T\wedge (\pi_{n-1}|_H)^*(\mu_{\cM})= \alpha \cdot T\wedge (\pi_{n}|_H)^*(\mu_{\cM})
\end{align}
for some constant $\alpha>0$, where $T=(\pi_{1}|_H)^*(\mu_{\cM})\wedge\cdots\wedge (\pi_{n-2}|_H)^*(\mu_{\cM})$ is an $(n-2,n-2)$-current on $H$. 

For $n=2$, equation \eqref{tobesliced} means that $(\pi_{1}|_H)^*(\mu_{\cM})= \alpha\cdot (\pi_{2}|_H)^*(\mu_{\cM})$ on the curve $H$ in $\C^2$.  But the projections are locally invertible away from finitely many points, so the measure equality induces a local isomorphism between a neighborhood of a point in $\del \cM \subset \C$ and its image.  Theorem \ref{Luo} then implies that this local isomorphism is the identity.  That is, the curve $H$ must be the diagonal line in $\C^2$ as claimed. 

Assume now that $n\ge 3$. 
Let $\pi: H\to \C^{n-2}$ be the projection to the first $n-2$ coordinates.  Observe that $T=\pi^{*}(\nu)$ for the measure $\nu=p^{*}_1(\mu_{\cM})\wedge \cdots\wedge p^*_{n-2}(\mu_{\cM})$ on $\C^{n-2}$, where $p_i:\C^{n-2}\to \C$ is the projection to the $i$-th coordinate.

 By our assumption, $\pi$ is dominant and the fiber-dimension theorem yields that the fibers $H_z:=H\cap\{x_1=z_1,\ldots,x_{n-2}=z_{n-2}\}\subset H$ are curves for  $z=(z_1,\ldots,z_n)$ in a Zariski open and dense subset of $\C^{n-2}$.  Note that each $(1,1)$-current  $p_j^*\mu_{\cM}$ has continuous potentials on $\C^{n-2}$, so the measure $\nu$ does not charge pluripolar sets. Thus the fiber $H_z$ is a curve for $\nu$-almost every $z$, and, by the characterization of slicing of currents as in \cite[Proposition 4.3]{Bassanelli:Berteloot}, equation \eqref{tobesliced} implies that 
\begin{align}\label{integrals agree}
\int_{\C^{n-2}}\left(\int_{H_z}\phi \, d\pi_{n-1}|_{H_z}^* (\mu_{\cM})(x) \right)d\nu (z)=\alpha \int_{\C^{n-2}}\left(\int_{H_z}\phi \, d\pi_{n}|_{H_z}^* (\mu_{\cM})(x) \right)d\nu(z),
\end{align}
for every continuous and compactly supported function $\phi$ on $H$.  It follows that we have equality of measures 
\begin{align}\label{slices equal}
\pi_{n-1}|_{H_z}^* (\mu_{\cM}) =\alpha\cdot  \pi_n|_{H_z}^*(\mu_{\cM})
\end{align}
for $\nu$-almost every $z:=(z_1,\ldots,z_{n-2})$ in $\C^{n-2}$. In detail, suppose there exists a point $z_0$ in the support of $\nu$ where $H_{z_0}$ is a curve and such that 
$$\pi_{n-1}|_{H_{z_0}}^* (\mu_{\cM}) \neq \alpha\cdot  \pi_n|_{H_{z_0}}^*(\mu_{\cM}).$$
Then we can find a continuous function $\psi_{z_0}$ on $H_{z_0}$ so that 
$$\int_{H_{z_0}}\psi_{z_0}\pi_{n-1}|_{H_{z_0}}^* (\mu_{\cM}) \neq \alpha \int_{H_{z_0}} \psi_{z_0} \pi_n|_{H_{z_0}}^*(\mu_{\cM}).$$
Note that the measures $\pi_{i}|_{H_{z}}^* (\mu_{\cM})$ vary continuously as functions on $z$ on a neighborhood of $z_0$, by the continuity of the potentials.  We thus infer that 
$$\int_{H_{z}}\psi_{z_0}\pi_{n-1}|_{H_{z}}^* (\mu_{\cM}) \neq \alpha \int_{H_{z}} \psi_{z_0} \pi_n|_{H_{z}}^*(\mu_{\cM}),$$
for all $z$ in a small open neighborhood $U$ of $z_0$. We can therefore find $\phi= h\cdot \psi_{z_0}$ where $h$ is a continuous function supported on $U$ and for which the equality \eqref{integrals agree} fails. 

Again by Theorem \ref{Luo}, equation \eqref{slices equal} yields that $H_z$ is special for $\nu$-almost all $z$. Since $T$ does not charge pluripolar sets and since $H$ projects dominantly on each $n-1$ coordinates, we infer that $H\subset \pi^{-1}_{(n-1,n)}(\Delta)$, where $\pi_{(i,j)}:\C^n\to \C^2$ is the projection to the $i$ and $j$-th coordinates. Repeating the argument using the equalities of all measures $\mu_j$, we get 
$$H\subset \bigcap_{i\neq j\in\{1,\ldots,n\}}\pi^{-1}_{(i,j)}(\Delta).$$
But since $H$ has dimension $n-1$ and $n\ge 3$ this is impossible. This completes our proof. 
\end{proof}

\subsection{Proof of Theorem \ref{eM^n}}
We can now complete the proof of Theorem \ref{eM^n} (and so also of Theorem \ref{M^n}) by reducing it to Proposition \ref{dominant hyper}.  This argument is inspired by \cite{Ghioca:Nguyen:Ye:C}. 

First we show that if Theorem \ref{eM^n} holds for hypersurfaces $X$, then it holds in general. 
So assume that the theorem is true when $X$ is a hypersurface and let $X$ be an irreducible subvariety of $\C^n$ with dimension $d<n-1$ which contains a generic sequence of small points. 
Permuting the coordinates if necessary, we may assume that $X$ projects dominantly to the first $d$ coordinates. 
Now let $\pi_{(j)}: \C^n\to \C^{d+1}$ denote the projection to the first $d$ and the $j$th coordinates. 
Let $X_j$ denote the Zariski closure of 
$\pi_{(j)}^{-1}(\pi_{(j)}(X))$ in $\C^n$. Each $X_j$ is a hypersurface in $\C^n$ and contains a generic sequence of small points. 
Therefore by our assumption $X_j$ must be special. 
If $X\subset X_j$ is special, then our claim follows. Otherwise $X_j=\{(x_1,\ldots,x_n)\in \C^n~:~x_j=c_j\}$ for a special point $c_j$ or $X_j=\{(x_1,\ldots,x_n)\in \C^n~:~x_j=x_k\}$ for some $k\in\{1,\ldots,d\}$. 
From the precise form of each $X_j$, it is easy to see that $\cap_{j=d+1}^nX_j$ has dimension $d=\dim X$. But $X\subset  \cap_{j=d+1}^nX_j$, so we must have $X=\cap_{j=d+1}^nX_j$ and our claim follows.

Therefore, it suffices to prove Theorem \ref{eM^n} for hypersurfaces $X$. 
Arguing by induction on $n$, we may further assume that $X$ projects dominantly on each $n-1$ coordinates. 
Indeed, if $n=2$ and the curve $X$ is vertical or horizontal then since it contains a generic sequence of small points, it must be special.  Assume now that $n\ge 3$, and that $X$ does not project dominantly on, say, the last $n-1$ coordinates. Then it has the form $X=\C\times X_0$ for a hypersurface $X_0\subset \C^{n-1}$ (see e.g. \cite[Lemma 3.1]{Mavraki:Schmidt:Wilms}). By induction, Theorem \ref{eM^n} follows by Proposition \ref{dominant hyper}. As explained in Remark \ref{qbar}, Theorem \ref{M^n} also holds. 

%%%%%%%%%

%%%%%%%%

\bigskip
\section{Maximal variation and the lower bound} \label{maximal}

In this section, we provide some basic background on the Chow variety $X_d$ of curves of degree $\leq d$ in $\C^2$, and we prove lower bounds on the number of special points in families of curves.

\subsection{Chow and maximal variation} \label{max section}
Fix integer $d\geq 1$.  We will work with the Chow variety $X_d$ of algebraic curves in the plane $\C^2$, defined over the field $\C$ of complex numbers, of degree $\leq d$.  As a variety, $X_d$ is simply the complement of a single point in a projective space $\P^{N_d}_\C$, where 
	$$N_d = \begin{pmatrix} d+2 \\ 2 \end{pmatrix}  - 1 \; =\;  \frac{d(d+3)}{2}.$$
Indeed, each curve is the vanishing locus of a nonzero homogeneous polynomial $F(x,y,z)$ of degree $d$, uniquely determined up to scale, and evaluated at points of the form $(x,y,1)$ for $(x,y) \in \C^2$.  We exclude the polynomial $F(x,y,z) = z^d$.  

Let $\cC \to V$ be a family of plane algebraic curves, parameterized by an algebraic variety $V$ defined over $\C$.  There is an induced map from $V$ to $X_d$ for some degree $d$.  We say that the family $\cC \to V$ is {\bf maximally varying} if the induced map $V \to X_d$ has finite fibers.  For each integer $m\geq 1$, we let
\begin{equation}\label{m power}
	\cC_V^m := \cC \times_V \cdots \times_V\cC
\end{equation}
denote the $m$-th fiber power of $\cC$ over $V$.  There is a natural map 
	$$\rho_m:  \cC_V^m \to \C^{2m}$$
defined by sending a tuple of $m$ points $x_1, \ldots, x_m$ on a curve $C \in V$ to the $m$-tuple $(x_1, \ldots, x_m)$ in $(\C^2)^m$.   

\begin{prop}  \label{dominant}
Suppose that $V$ is an irreducible quasiprojective complex algebraic variety of dimension $\ell \geq 1$.  If $\cC \to V$ is a maximally varying family of curves in $\C^2$ of degree $\leq d$, then the natural map $\rho_m$ is dominant for all $m\leq \ell$ and generically finite for $m = \ell$.
\end{prop}

\begin{proof}
The result is clear for $m = 1$, because the image of $\cC$ in $\C^2$ cannot be contained in a single algebraic curve if the image of $V$ in $X_d$ is not a point.  For $m>1$, it suffices to show that the image of $\rho_m$ contains the union of subvarieties of the form $\{(z_1, \ldots, z_{m-1})\} \times U_{(z_1, \ldots, z_{m-1})}$, where $U_{(z_1, \ldots, z_{m-1})}$ is Zariski open in $\C^2$, over a Zariski dense and open subset of points $(z_1, \ldots, z_{m-1}) \in (\C^2)^{m-1}$.   Indeed, the dominance follows because the maps are algebraic, and the generic finiteness for $m=\ell$ follows because $\dim \cC_V^m = \ell + m$.  

We proceed by induction.  We have already seen that the result holds for $m = 1$ and any $\ell\geq 1$.  Now assume $\ell > 1$, and fix $1 < m \leq \ell$.  Assume the result holds for $\rho_{m-1}$.  Then, as the smooth part $V^{sm}$ of $V$ is Zariski open and dense, the image of $\rho_{m-1}$ restricted to $\cC_V^{m-1}$ over $V^{sm}$ contains a Zariski open set $U \subset \C^{2(m-1)}$.  Choose any point $(z_1, \ldots, z_{m-1})$ in $U$.  Suppose that $\lambda_0 \in V^{sm}$ is a parameter for which $C_{\lambda_0}$ contains the points $z_1, \ldots, z_{m-1}$.  There is a subvariety $V_1$ of $V$ containing $\lambda_0$ and with codimension $\leq m-1$ consisting of curves $C_\lambda$ that persistently contain the points $z_1, \ldots, z_{m-1}$.  In particular, the dimension of $V_1$ is at least 1.  Maximal variation implies that the image of $\rho_1$ on $\cC$ over $V_1$ is dominant to $\C^2$.  It follows that $\rho_m$ on $\cC_{V_1}^m$ is dominant to $\{(z_1, \ldots, z_{m-1})\} \times \C^2$.  Letting the point $(z_1, \ldots, z_{m-1})$ vary over the image of $\rho_{m-1}$, we use the induction hypothesis to see that $\rho_m$ is dominant from $\cC_V^m$ to $\C^{2m}$.
\end{proof}

\subsection{A lower bound on the number of special points}
For a family of plane algebraic curves $\cC\to V$ parameterized by $V$, recall the definition of $\cC_V^m$ in \eqref{m power}. 

\begin{prop} \label{lower bound} 
Suppose that $\cC \to V$ is a maximally varying family of irreducible, complex algebraic curves in $\C^2$, over an irreducible quasiprojective complex algebraic variety $V$ of dimension $\ell \geq 1$.  Then the preimage $\rho_\ell^{-1}(S)$ of the set of special points $S \subset \C^{2\ell}$ is Zariski dense in $\cC_V^{\ell}$.  In particular, there is a Zariski dense set of curves $\lambda \in V$ for which the fiber $C_\lambda$ contains at least $\ell$ distinct special points of $\C^2$. \end{prop}

\begin{proof}
By maximal variation, we know that the fiber product $\cC_V^\ell$ maps generically finitely and dominantly by $\rho_\ell$ to $\C^{2\ell}$.  The special points are Zariski dense in the image.  This implies that the set of points $P = (\lambda, x_1, \ldots, x_\ell) \in \rho_\ell^{-1}(S) \subset \cC_V^\ell$, where $\lambda \in V$ and $\{x_1, \ldots, x_\ell\}$ is a collection of special points on the fiber $C_\lambda$ of $\cC$ over $\lambda$, is Zariski dense in $\cC_V^\ell$.  In particular, the $x_i$'s must be generally all distinct.
\end{proof}

%%%%%%%%%

\bigskip
\section{Optimal general upper bounds}
\label{upper}

In this section we prove Theorems \ref{uniformChow}, \ref{optimal}, \ref{sharp}, and \ref{lines}.
For each integer $d\geq 1$, let $X_d$ denote the Chow variety of all algebraic curves in $\C^2$ of degree $\leq d$ defined over $\C$. 

\subsection{The uniform bound of Theorem \ref{uniformChow}}
Let $\cC \to V$ denote a family of algebraic curves in $\C^2$, parameterized by an irreducible, quasiprojective variety $V$ over $\C$ of dimension $\ell \geq 1$, for which the general curve in the family is irreducible.  As introduced in \S\ref{max section}, there is a natural map 
	$$\rho_1: \cC \to \C^2,$$ 
sending each curve to its image in $\C^2$.  Recall the definitions of $\cC_V^m$ and 
	$$\rho_m: \cC_V^m \to \C^{2m}$$ 
given there, for each integer $m\geq 1$.  Recall also that the family is maximally varying if the induced map $V \to X_d$ has finite fibers.  From Proposition \ref{dominant}, we know that maximal variation implies that the maps $\rho_m$ are dominant for all $m \leq \dim V$.  

\begin{prop} \label{not special M}
Suppose that $\cC \to V$ is a maximally varying family of curves in $\C^2$ with $\ell = \dim V > 0$.  Assume the general curve in the family is irreducible.  Then the Zariski-closure of the image in $\C^{2(\ell+1)}$ of the fiber power $\cC_V^{\ell +1}$ by $\rho_{\ell+1}$ is not special, unless $\ell =1$ and $\cC$ is a family of horizontal or vertical lines in $\C^2$.
\end{prop}

As a consequence we have:

\begin{theorem} \label{uniform any}
For any family $\cC \to V$ of curves in $\C^2$ -- irreducible or not, maximally varying or not -- there is a uniform upper bound $M = M(\cC)$ on the number of special points on $C_\lambda$, for all $\lambda \in V$ over which the fiber $C_\lambda$ of $\cC$ has no special irreducible components.
\end{theorem}

\begin{proof}[Proof of Theorem \ref{uniform any}]
Assume Proposition \ref{not special M}.  If $V$ or the generic curve is reducible, we work with irreducible components.  If the family fails to be maximally varying, it is convenient to factor through the image of $V$ in the Chow variety $X_d$ for some degree $d$.  So we now assume that $V$ is an irreducible subvariety in $X_d$ of dimension $\geq 1$ and the associated curve family $\cC\to V$ consists of generally irreducible curves and not exclusively of horizontal or vertical lines.

Consider the fiber powers $\cC_V^m \to V$ for each $m\geq 1$.  Suppose there is a generic sequence of points $C_n \in V$ for which the number of special points of the curves $C_n \subset \C^2$ is larger than $n$, for each $n\in \N$.  This implies that the special points of $\C^{2m}$ are Zariski dense in the images $\rho_m(\cC_V^m)$ for every $m\geq 1$.  Indeed, this is clear for $m=1$ because $\rho_1$ maps $\cC$ dominantly to $\C^2$.  For each positive integer $m\geq 2$, the set of special points in $\rho_m(\cC_V^m)$ includes the $m$-tuples formed from the $n$ distinct special points on $C_n$; note that this set of $m$-tuples in $\C^{2m}$ is symmetric under permutation of the $m$ copies of $\C^2$.  Let $Z_m$ be the Zariski closure of these special points within $\rho_m(\cC_V^m)$; note that $Z_m$ will also be symmetric under permutation of the $m$ copies of $\C^2$.  Because $\{C_n\}$ is a generic sequence in $V$, note also that $\rho_m^{-1}(Z_m)$ must project dominantly to $V$.  If $Z_m$ is not equal to all of $\rho_m(\cC_V^m)$, then $\rho_m^{-1}(Z_m)$ is contained in a subvariety $H \subset \cC_V^m$ which is a family of hypersurfaces over $V$ that are symmetric with respect to permutation of the components in each fiber $C\times \cdots \times C$.  Now consider the projections from $\cC_V^m \to \cC_V^{m-1}$ forgetting one factor, restricted to the hypersurface $H$.  By the symmetry of $H$, each of these projections is generically finite and of the same degree, say $r(m)$.  So over a Zariski open subset of $V$, this bounds the number of points on a given curve $C$; in particular this contradicts the assumption on the sequence of curves $C_n$.  So the special points of $\C^{2m}$ must be Zariski dense in $\rho_m(\cC_V^m)$ for every $m\geq 1$.

From Theorem \ref{M^n}, the density of special points in $\rho_m(\cC_V^m)$ implies that (the Zariski closure) of $\rho_m(\cC_V^m)$ is special.  But taking $m = \dim V + 1$, this contradicts Proposition \ref{not special M}.

So there is a uniform bound on the number of special points in the curve $C \subset \C^2$ for all curves $C$ in a Zariski open subset $U$ of $V$.  We then repeat the argument on each of the finite set of irreducible components of $V\setminus U$.  We continue until we are left with families of vertical or horizontal lines.  
\end{proof}

\begin{proof}[Proof of Proposition \ref{not special M}]
From Proposition \ref{dominant}, we know that the map $\rho_m: \cC_V^m \to \C^{2m}$ is dominant for all $m \leq \ell$ and generically finite for $m = \ell$.  Note that $\dim \cC_V^{\ell+1} = 2\ell+1 < 2\ell + 2$, so the map 
	$$\rho_{\ell+1}: \cC_V^{\ell+1} \to \C^{2(\ell+1)}$$ 
cannot be dominant.  Consider the projections $\pi_{ij}$ from $\cC_V^{\ell+1}$ to $\C^2$ defined by composing $\rho_{\ell+1}$ with 
	$$(x_1, \ldots, x_{2\ell+2}) \mapsto (x_i,x_j)$$
for each pair $1 \leq i < j \leq 2\ell+2$.

Assume $\ell > 1$. The projections $\pi_{ij}$ are dominant for all pairs $i < j$, because they factor through the dominant maps $\cC_V^{\ell+1} \to \cC_V^{\ell} \to \C^{2 \ell}$, where the first arrow forgets the $k$-th factor of $\cC$ over $V$ for some choice of indices $\{2k-1, 2k\}$ not containing $i$ or $j$, and the second arrow is $\rho_\ell$.  In view of the structure of special subvarieties from Theorem \ref{M^n}, we see immediately that $\rho_{\ell+1}(\cC_V^{\ell + 1})$ cannot be special in $\C^{2\ell+2}$.  

Now suppose that $\ell = 1$, and assume that $\cC$ is not a family of vertical or horizontal lines.  We aim to show that $\rho_2(\cC_V^2)$ is not a special hypersurface in $\mathbb{C}^4$.  Note that a general curve $C$ in the family $\cC$ projects dominantly to both coordinates in $\C^2$. It follows that $\rho_2(\cC_V^2)$ cannot be contained in the hyperplane $\{x_i = c_i\}$ for a special parameter $c_i$ and any $i \in \{1, 2, 3, 4\}$.  Because the curves over $V$ are not all equal to the diagonal line in $\C^2$, the space $\rho_2(\cC_V^2)$ also cannot be contained in the special hypersurfaces defined by $\{x_k =  x_{k+1}\}$ for $k \in \{1,3\}$.  Recalling the definition of special subvarieties, it remains to check that $\rho_2(\cC_V^2)$ does not lie in any of the hypersurfaces $\{x_1 = x_3\}$, $\{x_1=x_4\}$, $\{x_2=x_3\}$, or $\{x_2 = x_4\}$.  But for a general choice of curve $C$ in the family, the product $C \times C \subset \C^4$ maps dominantly to the spaces of pairs with coordinates $(x_1, x_3)$, $(x_2, x_4)$, $(x_2,x_3)$, or $(x_1,x_4)$.  This proves that $\rho_2(\cC_V^2)$ is not special.

Finally suppose that $\cC$ is a family of vertical or horizontal lines.  For concreteness, we can take $V = \C$ and $\lambda \in V$ corresponding to the vertical line $\{x = \lambda\}$ for $\lambda \in V$.  Then the image of $\cC_V^2$ in $\C^4$ is the set of all 4-tuples $(\lambda, x_2, \lambda, x_4)$ for any $(\lambda, x_2, x_4) \in \C^3$.  In other words, the image of $\cC_V^2$ is the special hypersurface defined by $\{x_1 = x_3\}$.  Similarly for families of horizontal lines.
\end{proof}

\begin{proof}[Proof of Theorem \ref{uniformChow}]
The theorem is an immediate consequence of Theorem \ref{uniform any}, taking $V = X_d$.
\end{proof}

\subsection{Optimal general bound over Chow; proof of Theorems \ref{optimal} and \ref{sharp}}  \label{proof of optimal}
Let $V = X_d$ be the Chow variety of all affine curves of degree $\leq d$ in $\C^2$ and $\cC\to V$ the universal family of such curves.  Recall from \S\ref{max section} that 
	$$N_d := \dim V  =  \frac{d(d+3)}{2}.$$ 
Consider the fiber power $\cC_V^{N_d+1} \to V$ and its image under the natural map
	$$\rho := \rho_{N_d+1} : \cC_V^{N_d+1} \to \C^{2(N_d+1)}.$$
	
Suppose that $S$ is a special subvariety of $\C^{2(N_d+1)}$ that is contained in the Zariski closure of the image $\rho(\cC_V^{N_d+1})$ for which $\rho^{-1}(S)$ projects dominantly to $V$.  We aim to show that $S$ must lie in the union of special diagonals
\begin{equation} \label{Dij}
	\Delta_{i,j} := \{(x_1, \ldots, x_{2N_d+2}) \in \C^{2N_d+2}:  (x_i, x_{i+1}) = (x_j, x_{j+1})\}
\end{equation}
for odd integers $i$ and $j$ satisfying $1\leq i < j \leq 2N_d +1$.  

This classification will imply the two theorems.  Indeed, if there were a generic sequence of elements $C_n \in V$ for which the curve $C_n \subset \C^2$ contains at least $N_d+1$ distinct special points of $\C^2$, then the $(N_d+1)$-tuples of such points will be special in $\C^{2(N_d+1)}$ and will lie outside of the special diagonals $\Delta_{i,j}$.  From Theorem \ref{M^n}, each irreducible component $Z$ of the Zariski closure of these special points in $\C^{2(N_d+1)}$ is itself a special subvariety, and by construction is contained in the closure of $\rho(\cC_V^{N_d+1})$.  As $\{C_n\}$ is a generic sequence of points in $V$, the preimage $\rho^{-1}(Z)$ of each component will project dominantly to $V$.  In other words, this $Z$ is a special subvariety of the type described, but not contained in the special diagonals $\Delta_{i,j}$, leading to a contradiction.  This will prove Theorem \ref{optimal}.  The equality of Theorem \ref{sharp} is then a consequence of the lower bound in Proposition \ref{lower bound}.

For the proof, suppose that $S$ is a special subvariety of $\C^{2(N_d+1)}$ that is contained in the Zariski closure of the image $\rho(\cC_V^{N_d+1})$ for which $\rho^{-1}(S)$ projects dominantly to $V$. Recall that our goal is to show that $S$ must lie in the union of the special diagonals $\Delta_{i,j}$. We begin with a few important observations.  First, note that $\dim \rho^{-1}(S) \geq \dim S$, so the dominance of the projection to $V$ implies that a general fiber of this projection has dimension $\geq \dim S - N_d$.  In other words, the intersection of $S$ with $C \times \cdots \times C$ in $\C^{2(N_d+1)}$ has dimension at least 
	$$\dim S - N_d = N_d + 2 - \codim S,$$
for a general curve $C$ in $V$.  Moreover, as the image $\rho(\cC_V^{N_d+1})$ is not itself special in $\C^{2(N_d+1)}$ by Proposition \ref{not special M}, we have that $S\subsetneq \overline{\rho(\cC_V^{N_d+1})} \subsetneq \C^{2(N_d+1)}$. Therefore, the codimension of $S$ in $\C^{2(N_d+1)}$ will be at least 2.  We begin by working case by case through some examples of special subvarieties, as classified in Theorem \ref{M^n}, to see that they either cannot be contained in $\overline{\rho(\cC_V^{N_d+1})}$ or that $\rho^{-1}(S)$ cannot project dominantly to $V$, unless $S$ is contained in one of the special diagonals. We then handle the general case.

\begin{itemize}
\item $S = \{x_1 = c_1 \mbox{ and } x_2 = c_2\}$:  a general curve $C\in V$ does not pass through the point $(c_1, c_2)\in \C^2$, so $\rho^{-1}(S)$ cannot project dominantly to $V$. 

\item $S = \{x_1 = c_1 \mbox{ and } x_3 = c_3\}$:  A general curve in $\C^2$ of degree $d$ projects dominantly to its 1st coordinate, so $\rho^{-1}(S)$ does project dominantly to $V$ in this case.  However, the intersection of $S$ with a general fiber $C \times \cdots \times C \subset \C^{2N_d+2}$ will have dimension only $N_d-1 < \dim S-N_d$, so this $S$ could not have been contained in the closure of $\rho(\cC_V^{N_d+1})$.  

%A similar argument applies if we specify a single coordinate for each of $m$ points, for any $2 \leq m \leq N_d+1$.  

\item $S = \{x_1 = c_1 \mbox{ and } x_2 = x_3\}$:  Again this $S$ has codimension 2, while the intersection with $C\times \cdots \times C$ for a general curve $C \in V$ has dimension only $N_d-1$.  

\item $S= \Delta_{1,3} = \{x_1 = x_3 \mbox{ and } x_2 = x_4\}$:  These relations again impose conditions on two of the $N_d+1$ components of $C \times \cdots \times C$.  But note that any collection of $N_d+1$ points $(x_1, x_2), \ldots, (x_{2N_d+1}, x_{2N_d+2})$ in $\C^2$ satisfying $(x_1,x_2) = (x_3, x_4)$ will lie on some curve of degree $d$, because at most $N_d$ of the points are distinct.  So all of $\Delta_{1,3}$ is contained in $\rho(\cC_V^{N_d+1})$.  The intersection of $\Delta_{1,3}$ with a general $C \times \cdots \times C$ has dimension $N_d$, which is the expected dimension.  

\item $S = \{x_1 = c_1 \mbox{ and } x_2 = x_3 = x_4\}$:  Again we impose relations on only two of the $N_d+1$ components of $C \times \cdots \times C$, but there are too many relations; a general curve will not intersect both $(c_1, y)$ and $(y, y)$ for any choice of $y \in \C$.  Consequently, the preimage $\rho^{-1}(S)$ in $\cC_V^{N_d+1}$ will not project dominantly to $V$. 

\item $S = \{x_1 = c_1 \mbox{ and } x_2 = x_3 \mbox{ and }  x_4 = c_4\}$:  These three relations are imposed upon only two of the $N_d+1$ components of $C \times \cdots \times C$.  As in the previous example, a general curve $C$ will not intersect both $(c_1, y)$ and $(y, c_4)$ for any choice of $y \in \C$.  The preimage $\rho^{-1}(S)$ in $\cC_V^{N_d+1}$ will not project dominantly to $V$. 
\end{itemize}

In general, recall that since $S$ is special it is defined by imposing `special relations' of the form $x_i=c_i$ for a PCF parameter $c_i$ or $x_k=x_{\ell}$ for $i,k,\ell\in\{1,\ldots,2N_d+2\}$.  In general, we see that if we define $S$ by imposing up to $N_d+1$ special relations on the coordinates of the $N_d+1$ components of a general product $C\times \cdots \times C$ in $\C^{2(N_d+1)}$, as long as no one point is constant (as in the first example) nor that there are three relations on coordinates of two components (as in the last two examples), nor that two of the coordinates are required to agree (so as to be a subvariety of a special diagonal $\Delta_{i,j}$), then the preimage $\rho^{-1}(S)$ in $\cC_V^{N_d+1}$ would project dominantly to $V$, but the general intersection of $S$ with $C \times \cdots \times C$ will not have sufficiently large dimension.  That is, this $S$ will not lie in $\overline{\rho(\cC_V^{N_d+1})}$ in $\C^{2(N_d+1)}$.  If we specify that one of the components is a fixed special point, or if there are at least three relations imposed upon a pair of components of $C\times \cdots \times C$ (as will be the case if $\codim S > N_d+1$), then the preimage $\rho^{-1}(S)$ will not project dominantly to $V$.  
%The unique exception is when the special variety lies in one of the diagonals $\Delta_{i,j}$.  
This completes the proofs of Theorems \ref{optimal} and \ref{sharp}.

\subsection{Lines and the proof of Theorem \ref{lines}}  \label{M lines}
The classification of special subvarieties (Theorem \ref{M^n}) and the proof strategy in \S\ref{proof of optimal} for Theorem \ref{optimal} suggest that the general curve in a ``generically chosen" maximally-varying family $\cC \to V$ of curves in $\C^2$ with dimension $\ell = \dim V$, in any degree, will intersect at most $\ell$ distinct special points.  Here we show this is indeed the case in degree $d = 1$.  Before doing so, we give an example where this expectation fails.

\begin{example}\label{pencil} [An exceptional family of lines] Consider the pencil of lines in $\C^2$ passing through a given special point $P$, parameterized by $V \iso \P^1$; for example take $P = (-1, -2) \in \C^2$.  Because we can connect any special point in $\C^2$ to $P$ with a line, there are infinitely many lines in this family containing at least 2 special points, though $\dim V = 1$.  
\end{example}

Less obvious is the fact that the general bound on the number of special points in a line, for the family of lines in Example \ref{pencil}, is also 2.  That is, there are at most finitely many lines in $\C^2$ through the special point $P$ containing more than 2 distinct special points of $\C^2$.  On the other hand, if we consider the pencil of lines in $\C^2$ passing through a non-special point such as $P = (1,1)$, then all but finitely many lines in the family will have at most 1 special point.  These facts are contained in the following proposition:

%We know from Theorems \ref{optimal} and \ref{sharp} that the sharp bound on the number of special points in a general line in $\C^2$ is 2.  The proofs of Theorems \ref{optimal} and \ref{sharp} told us that the only special subvarieties of $\C^6$ that lie in the image of $\cC_{X_1}^3$ and dominate $X_1$ are the diagonals $\Delta_{1,3}$, $\Delta_{1,5}$ and $\Delta_{3,5}$.  But Theorem \ref{optimal} leaves open the possibility that there are finitely many curves in $X_1$ defining families where infinitely many lines contain $\geq 3$ special points.  For example, we know that the families of horizontal and vertical lines are among these.   However, we will now prove that the horizontal and vertical lines are the only exceptional families where infinitely many lines contain $\geq 3$ special points.  

\begin{prop} \label{prop M lines}
Let $V \subset X_1$ be an irreducible curve in the Chow variety of lines in $\C^2$, not consisting exclusively of vertical or horizontal lines.  Then, outside of finitely many parameters $\lambda \in V$, the lines of the family will intersect at most 1 special point in $\C^2$, unless $V$ is
\begin{enumerate}
\item the family of all lines through a special point $P = (p_1, p_2) \in \C^2$; 
\item the family of lines defined by $L_\lambda = \{(x,y) \in \C^2: x + y = \lambda\}$, for $\lambda \in \C$; 
\item the family of lines $L_\lambda$ containing $(c_1, \lambda)$ and $(\lambda, c_2)$ for special parameters $c_1 \not= c_2$, for $\lambda \in \C$.
\end{enumerate}
In each of these 3 cases, outside of a finite set of parameters $\lambda \in V$, there will be at most 2 special points on the line $L_\lambda$; moreover, there are infinitely many parameters $\lambda \in V$ for which the line $L_\lambda$ contains exactly 2 special points of $\C^2$.
\end{prop}

\begin{proof}
Let $V \subset X_1$ be an irreducible algebraic curve defined over $\C$, not consisting of vertical or horizontal lines.  Let $\cC \to V$ denote this family of lines over $V$.  Consider the special points in $\rho_2(\cC_V^2)$ in $\C^4$, for the map $\rho_2$ defined in \S\ref{max section}.  Note that their Zariski closure must contain the special diagonal surface
	$$\Delta_{1,3} = \{(x_1, x_2, x_3, x_4)\in \mathbb{C}^4:  (x_1,x_2) = (x_3, x_4)\},$$
because $\rho_1$ is dominant to $\C^2$ (in which special points are Zariski dense).  The general bound on the number of special points on lines $L \in V$ will be 1 unless either 
\begin{enumerate}
\item[(S)] there is a special surface contained in the Zariski closure $\overline{\rho_2(\cC_V^2)}$, other than the diagonal $\Delta_{1,3}$, intersecting $L \times L$ in a curve for general $L \in V$; or 
\item[(C)]  there is a special curve contained in the Zariski closure $\overline{\rho_2(\cC_V^2)}$, not contained in $\Delta_{1,3}$, intersecting $L \times L$ in a nonempty finite set for general $L \in V$.
\end{enumerate}
Indeed, if an infinite collection of lines $L$ in $V$ had at least 2 special points, then the Zariski closure of those pairs of points would not be contained in $\Delta_{1,3}$ and would form a special subvariety lying in $\overline{\rho_2(\cC_V^2)}$.   Some irreducible component $Z$ of this Zariski closure will be positive dimensional (because it contains an infinite collection of points), and it must have dimension $< \dim \overline{\rho_2(\cC_V^2)} = 3$ because the hypersurface $\overline{\rho_2(\cC_V^2)}$ in $\C^4$ cannot be special by Proposition \ref{not special M}.  Thus, $Z$ is either a curve or a surface.  Because the collection of lines $L$ containing these special points was infinite in the curve $V$, some component $Z$ must have preimage $\rho_2^{-1}(Z)$ that projects dominantly to $V$.  The dimension of the intersection of $Z$ with the general $L\times L$, as described in cases (S) and (C), then follows by dimension count, exactly as in \S\ref{proof of optimal}.

We will see that cases (1) and (2) of the proposition correspond to the existence of special surfaces of type (S), and case (3) of the proposition gives rise to special curves of type (C).  For each of the families (1), (2), and (3), it is clear that there are infinitely many lines in the family containing at least 2 distinct special points.  It will remain to show that there are at most 2 special points on all but finitely many lines in each of these families.

We work case by case, considering each type of special surface or curve in $\C^4$:

\begin{itemize}
\item[(S1)] $\{x \in \C^4: x_i=c_i \mbox{ and } x_j =  c_j\}$ for $i< j$ in $\{1, 2, 3, 4\}$:  If $\{i,j\}$ is $\{1,2\}$ or $\{3,4\}$, then this special surface is contained in $\overline{\rho_2(\cC_V^2)}$ if and only if $V$ is the pencil of lines through the special point $P = (c_i, c_j)$, and we denote these surfaces by  
	$$S_{P,1} := \{P\} \times \C^2 \; \mbox{ and } \; S_{P,2} := \C^2 \times \{P\}.$$
If $\{i,j\}$ is $\{1,3\}$, $\{2,4\}$, $\{1,4\}$, or $\{2,3\}$, then the intersection with $L \times L$ is finite for general $L \in V$ so this surface cannot be of type (S).

\item[(S2)] $\{x \in \C^4: x_i = c_i \mbox{ and } x_j = x_k\}$ for three distinct indices $i,j,k$:    If $\{j,k\}$ is $\{1,2\}$ or $\{3,4\}$, then the surface is not contained in $\overline{\rho_2(\cC_V^2)}$ because the intersections of $L$ with $\{x=y\}$ and $\{x= c_i\}$ are finite for general $L \in V$.  If $\{i,j\} = \{1,2\}$, the intersection of the special surface with $L \times L$ is again generally finite; other cases are similar, and none can be of type (S).

\item[(S3)] $\{x_i = x_j \mbox{ and } x_k = x_m\}$ with disjoint pairs of indices $\{i,j\}$ and $\{k, m\}$ in $\{1,2,3,4\}$:  If the pairs are $\{1,2\}$ and $\{3,4\}$, then the intersection with $L  \times L$ is finite for general $L \in V$, so it cannot be of type (S).  If the pairs are $\{1,3\}$ and $\{2,4\}$, then the special subvariety is the special diagonal surface
	$$\Delta_{1,3} = \{(x_1,x_2) = (x_3, x_4)\}.$$
If they are $\{1,4\}$ and $\{2,3\}$, then the the surface lies in $\overline{\rho_2(\cC_V^2)}$ if and only if the points on $L$ come in symmetric pairs (so $(x,y) \in L$ $\iff$ $(y,x) \in L$) for general $L \in V$.  In other words, the family of lines is of the form 
	$$x + y = b$$
for a nonconstant function $b$ on $V$, and we will denote this special surface by
	$$D := \{x_1 = x_4 \mbox{ and } x_2 = x_3\}.$$
For this family of lines, a point $(x,y)$ on the line will be special if and only if $(y,x)$ is special, so there are infinitely many such lines with at least two distinct special points.

\item[(S4)] $\{x_i = x_j = x_k\}$ for three distinct indices $i, j, k$:   There is at least one pair of indices which is either $\{1,2\}$ or $\{3,4\}$.  But then the intersection with a $L \times L$ is finite for general $L \in V$, so this surface cannot be of type (S).
\end{itemize}

We now consider the existence of special curves of type (C).  We work case by case again, considering each type of special curve in $\C^4$. 
 
\begin{itemize}
\item[(C1)] $\{x \in \C^4: x_i=c_i,  \; x_j =  c_j, \; x_k = c_k\}$ for $i< j < k$ in $\{1, 2, 3, 4\}$:  There is a pair of indices equal to $\{1,2\}$ or $\{3,4\}$.  So the product $L \times L$ intersects this curve for a general $L \in V$ if and only if the family of lines persistently contains a special point $P$.  In other words, the family must be case (1) of the proposition, and the Zariski closure $\overline{\rho_2(\cC_V^2)}$ also contains the surfaces $S_{P,1}$ and $S_{P,2}$ of type (S1). 

\item[(C2)] $\{x_i=c_i,  \; x_j =  c_j, \; x_k = x_m\}$ for disjoint pairs $\{i,j\}$ and $\{k,m\}$:  If $\{i,j\} = \{1,2\}$ then the general intersection with $L \times L$ is empty unless the lines contain $P = (c_1, c_2)$ for all $L \in V$.  In particular, this family of lines must be case (1) of the proposition, and the Zariski closure $\overline{\rho_2(\cC_V^2)}$ also contains the surfaces $S_{P,1}$ and $S_{P,2}$ of type (S1). Similarly for $\{i,j\} = \{3,4\}$. If $\{i,j\} = \{1,3\}$, then the equality $x_2 = x_4$ in $L\times L$ would imply that $c_1 = c_3$ because $L$ is degree 1 and not horizontal for all $L$, making this special curve lie in the diagonal $\Delta_{1,3}$.  Similarly for $\{i,j\} = \{2,4\}$.  So, for each of these cases, the curve cannot be of type (C).  

For $\{i,j\} = \{1,4\}$ with $c_1 = c_4$, the relation $x_2 = x_3$ implies that the general line in the family must intersect points of the form $(c_1, y)$ and $(y, c_1)$ for some $y \in \C$ (where the $y$ value can vary with the line).  If $y = c_1$ for a general line, then the family of lines must be case (1) with $P = (c_1, c_1)$, and the Zariski closure $\overline{\rho_2(\cC_V^2)}$ also contains the surfaces $S_{P,1}$ and $S_{P,2}$ of type (S1).  If $y \not= c_1$ for a general line in the family, then the pair of points $(c_1, y)$ and $(y, c_1)$ determine the line uniquely; the family must be of the form $x+y=\lambda$, and the Zariski closure $\overline{\rho_2(\cC_V^2)}$ also contains the surface $D$ of type (S3).  For $\{i,j\} = \{1,4\}$ with $c_1 \not= c_4$, the relation $x_2 = x_3$ implies that the general line in the family must intersect points of the form $(c_1, y)$ and $(y, c_4)$ for some $y \in \C$ (where the $y$ value must vary with the line).  The family of lines is therefore case (3) of the proposition, and the Zariski closure $\overline{\rho_2(\cC_V^2)}$ will contain the curve of type (C) defined by
	$$C_{c_1,c_4, 1}  := \{x_1 = c_1 \mbox{ and } x_4 = c_4 \mbox{ and } x_2 = x_3\}.$$
By the symmetry of $\overline{\rho_2(\cC_V^2)}$, we see that the curve
	$$C_{c_1,c_4, 2}  = \{x_2 = c_4 \mbox{ and } x_3 = c_1 \mbox{ and } x_1 = x_4\}$$
will also be contained in $\overline{\rho_2(\cC_V^2)}$, for the same family of lines.   The case of $\{i,j\} = \{2,3\}$ leads to the same conclusion.

\item[(C3)]   $\{x_i=c_i,  x_j = x_k = x_m\}$:  Assume first that $i = 1$.  The relations determine a curve of type (C) if the lines of the family contain both $(c_1, y)$ and $(y,y)$ for some $y \in \C$ (where the value of $y$ can vary with the line).  If the general line intersects the point $P = (c_1, c_1)$, then the family is case (1), and the Zariski closure $\overline{\rho_2(\cC_V^2)}$ also contains the surfaces $S_{P,1}$ and $S_{P,2}$ of type (S1).  If $y \not= c_1$ for a general line, then the lines must be horizontal, and this case has been ruled out by assumption. Similarly for $i = 2, 3, 4$.  

\item[(C4)]  $\{x_1 = x_2= x_3 = x_4\}$:  This curve is contained in the diagonal surface $\Delta_{1,3}$.   
\end{itemize}

The case-by-case analysis shows that the three types of families of lines listed in the proposition are the only families of lines that give rise to special subvarieties of type (S) or type (C). Thus, any other 1-parameter maximally-varying family of lines will have at most 1 special point on the general line in the family.  

To see that the bound is at most 2 on the three types of exceptions (1), (2), and (3), we look again at the cases and the structure of the Zariski closure of the special points in $\rho_2(\cC_V^2)$.  Suppose we are in case (1), and assume there are at least 3 distinct special points on infinitely many lines $L$ in the family.  Then, considering pairs of special points on a line $L$, with neither equal to the given special $P \in \C^2$, we build a component of the Zariski closure of special points in $\rho_2(\cC_V^2)$ that is neither in $\Delta_{1,3}$ nor in the surfaces $S_{P,1}$ or $S_{P,2}$ of type (S1) above.  Similarly for case (2), choosing pairs of distinct special points that are not symmetric (as $(x,y)$ and $(y,x)$) leads to a component of the Zariski closure of special points in $\overline{\rho_2(\cC_V^2)}$ that is neither in $\Delta_{1,3}$ nor in the surface $D$ of type (S3).  And finally, for case (3), the existence of pairs of points that are distinct and not equal to the pair $(c_1, \lambda)$ and $(\lambda, c_4)$ as described in case (C2) leads to a special component not contained in $\Delta_{1,3}$ nor in the curves $C_{c_1, c_4, 1}$ or $C_{c_1, c_4, 2}$.  

Thus, it remains to observe that a family of lines $\cC\to V$ over a curve $V$ cannot be exhibited as a family of the form (1), (2), or (3) in two distinct ways.  For example, as there is a unique line through distinct points $P$ and $Q$ in $\C^2$, there cannot be a family of lines exhibited as case (1) of the proposition for two distinct special points $P \not= Q$.  It is also clear that there is no point $P \in \C^2$ in every line of the family of case (2), so cases (1) and (2) cannot coincide.  Given a family as in case (3), we can easily compute that there are at most two lines in the family containing any given point $P\in \C^2$, so cases (3) and (1) cannot coincide.   The family of lines in case (2) has constant slope, while those of case (3) have slopes varying with $\lambda$ because $c_1 \not=c_2$, so the cases (2) and (3) cannot coincide.  Finally, we check that a family of lines cannot be exhibited as case (3) for distinct special pairs $(c_1, c_2)$ and $(c_1', c_2')$ with $c_1 \not= c_2$ and $c_1' \not= c_2'$.  If so, there would be a quadratic relation that must be satisfied for all parameters $\lambda \in \C$, namely
	$$c_1 \lambda^2 - c_2 \lambda^2  - \lambda^2 c_1' + \lambda^2 c_2'  + 2 c_2 c_1' \lambda  - 2 c_1c_2' \lambda +  c_1^2 c_2'  - c_1^2 c_2 + c_1 c_2^2   - c_2^2 c_1' = 0,$$
which implies that $(c_1, c_2) = (c_1', c_2')$.  

This proves that the bound is at most 2 for each of these exceptional families.  To see that the bound is optimal, we observe that the families are constructed to have infinitely many lines containing at least 2 special points.  
\end{proof}

\begin{proof}[Proof of Theorem \ref{lines}]
The Chow variety of lines $X_1$ has dimension 2.  From Theorem \ref{optimal}, we know that there is a finite union $V_1$ of irreducible curves and points in $X_1$ so that there are at most 2 special points on each line $L \not\in V_1$.  Now fix an irreducible curve $C \subset V_1$, and assume it is not the family of vertical or horizontal lines in $\C^2$.  From Proposition \ref{prop M lines}, there is again a bound of 2 on the number of special points for all but finitely many $L \in C$.  This completes the proof.
\end{proof}

%%%%%%
\bigskip
\section{Real algebraic curves in $\R^2$}
\label{real lines}

In this section, we observe that Theorems \ref{uniformChow}, \ref{optimal}, \ref{sharp}, and \ref{lines} apply to real algebraic curves in $\R^2$ passing through PCF parameters in the Mandelbrot set, in particular providing a proof of Theorem \ref{real}.

Suppose $P(x,y) \in \R[x,y]$ is a polynomial with degree $d \geq 1$.  Writing $x = \frac12 (c + \bar{c})$ and $y = \frac{1}{2i}(c - \bar{c})$, we obtain a polynomial of $c$ and $\bar{c}$ of degree $d$ with complex coefficients.  In this way, any real algebraic curve in $\R^2$ passing through a collection $\{c_1, \ldots, c_m\}$ of PCF parameters in $\C$ gives rise to a complex algebraic curve in $\C^2$ passing through special points $\{(c_1, \bar{c}_1), \ldots, (c_m, \bar{c}_m)\}$.  (Recall that the set of special parameters is symmetric under complex conjugation.)  

For example, if we begin with the line in $\R^2$ defined by
	$$\left\{(x,y) \in \R^2:  a \, x + b\, y = r\right\}$$
with $a, b, r \in \R$, then this line contains $c \in \C$ if and only if the complex line 
	$$\left\{(x,y) \in \C^2:  \frac12 (a- i b) \, x + \frac12 (a + i b) \, y = r\right\}$$
contains the point $(c, \bar{c})$ in $\C^2$.  In particular, taking $b=1$ and $a=r = 0$ shows that the real axis in $\C$ corresponds to the diagonal line $x=y$ in $\C^2$.  Note that the vertical and horizontal lines in $\C^2$ cannot arise by this construction. 

\begin{example} The imaginary axis in $\C$ contains the three PCF parameters $\{i, 0, -i\}$ and corresponds to the complex line $y = -x$ in $\C^2$.  Other than the real and imaginary axes in $\C$, we do not know any examples of real lines with more than two PCF parameters.
\end{example}

We see immediately that Theorem \ref{uniformChow} applies to real algebraic curves in each degree $d\geq 1$, implying there is a uniform bound on the number of PCF parameters on any such curve, depending only on the degree.  And so does Theorem \ref{lines}, as the real axis in $\C$ is the only line containing infinitely many PCF parameters, implying that there are only finitely many real lines in $\C$ passing through more than 2 PCF parameters.  This completes the proof of Theorem \ref{real}.

Note that the set of all complex algebraic curves of degree $d$ built from real curves in the above way is Zariski-dense in the Chow variety $X_d$ of all complex curves of degree $\leq d$ in $\C^2$.  Therefore, Theorem \ref{optimal} also holds for real algebraic curves.  Finally, observing that there always exists a real algebraic curve of degree $d$ through any collection of $d(d+3)/2$ points in $\C$, we see that Theorem \ref{sharp} also holds by choosing the curves to pass through collections of PCF parameters, so the bound of $d(d+3)/2$ is optimal for a general real curve in degree $d$.  

\begin{example}  Holly Krieger pointed out to us that Theorem \ref{lines} also implies there are only finitely many horizontal real lines in $\C$ that contain more than 1 PCF parameter.  Indeed, if two distinct PCF parameters, say $c_1$ and $c_2$, have the same nonzero imaginary part, then $c_2 - c_1 = \bar{c}_2 - \bar{c}_1 \in \R$, and the four pairs $(c_1, \bar{c}_2)$, $(\bar{c}_2, c_1)$, $(c_2, \bar{c}_1)$, and $(\bar{c}_1, c_2)$ are on the complex line
	$$x + y = \alpha := c_1 + \bar{c}_2$$
in $\C^2$.  We do not know of any examples of horizontal lines, other than the real axis in $\C$, containing more than one PCF parameter.
\end{example}

%%%%%%%%%

\bigskip \bigskip
%\bibliographystyle{../tex/bib/math}
%\bibliography{../tex/bib/math}

\begin{thebibliography}{GKNY}

\bibitem[An]{Andre:finitude}
Y.~Andr{\'e}.
\newblock {Finitude des couples d'invariants modulaires singuliers sur une
  courbe alg\'ebrique plane non modulaire}.
\newblock \textit{J. Reine Angew. Math.} {\bf 505}(1998), 203--208.

\bibitem[BD]{BD:preperiodic}
Matthew Baker and Laura DeMarco.
\newblock {Preperiodic points and unlikely intersections}.
\newblock \textit{Duke Math. J.} {\bf 159}(2011), 1--29.

%\bibitem[BH]{Baker:Hsia}
%Matthew H.~Baker and Liang-Chung Hsia.
%\newblock {Canonical heights, transfinite diameters, and polynomial dynamics}.
%\newblock \textit{ J. Reine Angew. Math.} {\bf 585}(2005), 61--92.
%
\bibitem[BB]{Bassanelli:Berteloot}
Giovanni Bassanelli and Fran\c{c}ois Berteloot.
\newblock {Bifurcation currents in holomorphic dynamics on {$\mathbb{ P}^k$}}.
\newblock \textit{J. Reine Angew. Math.} {\bf 608}(2007), 201--235.

\bibitem[BLM]{Bilu:Luca:Masser}
Yuri Bilu, Florian Luca, and David Masser.
\newblock {Collinear CM-points}.
\newblock \textit{Algebra Number Theory} {\bf 11}(2017), 1047--1087.

\bibitem[Bo]{Bogomolov:original}
F.~A. Bogomolov.
\newblock {Points of finite order on an abelian variety}.
\newblock \textit{Izv. Akad. Nauk SSSR Ser. Mat.} (1980), 782--804, 973.

\bibitem[CS]{Call:Silverman}
Gregory~S. Call and Joseph~H. Silverman.
\newblock {Canonical heights on varieties with morphisms}.
\newblock \textit{Compositio Math.} {\bf 89}(1993), 163--205.

\bibitem[De]{D:current}
Laura DeMarco.
\newblock {Dynamics of rational maps: a current on the bifurcation locus}.
\newblock \textit{Math. Res. Lett.} {\bf 8}(2001), 57--66.

\bibitem[Ed1]{Edixhoven:special2}
Bas Edixhoven.
\newblock {Special points on the product of two modular curves}.
\newblock \textit{Compositio Math.} {\bf 114}(1998), 315--328.

\bibitem[Ed2]{Edixhoven:special}
Bas Edixhoven.
\newblock {Special points on products of modular curves}.
\newblock \textit{Duke Math. J.} {\bf 126}(2005), 325--348.

\bibitem[FG]{Favre:Gauthier:book}
Charles Favre and Thomas Gauthier.
\newblock \textit{The Arithmetic of Polynomial Dynamical Pairs}, volume 214 of
  { Annals of Mathematics Studies}.
\newblock Princeton University Press, Princeton, NJ, 2022.

\bibitem[FRL]{FRL:quantitative}
Charles Favre and Juan Rivera-Letelier.
\newblock{\'{E}quidistribution quantitative des points de petite hauteur sur
              la droite projective},
\newblock\textit{Mathematische Annalen},
{\bf 335}, (2006), no.2, 311--361.

\bibitem[GHT]{GHT:rational}
D.~Ghioca, L.-C. Hsia, and T.~Tucker.
\newblock {Preperiodic points for families of rational maps}.
\newblock \textit{ Proc. London Math. Soc.} {\bf 110}(2015), 395--427.

\bibitem[GKNY]{GKNY}
D.~Ghioca, H.~Krieger, K.~D. Nguyen, and H.~Ye.
\newblock {The dynamical {A}ndr\'{e}-{O}ort conjecture: unicritical
  polynomials}.
\newblock \textit{Duke Math. J.} {\bf 166}(2017), 1--25.

\bibitem[GNY]{Ghioca:Nguyen:Ye:C}
Dragos Ghioca, Khoa~D. Nguyen, and Hexi Ye.
\newblock {The dynamical {M}anin-{M}umford conjecture and the dynamical
  {B}ogomolov conjecture for endomorphisms of {$(\Bbb P^1)^n$}}.
\newblock \textit{Compos. Math.} {\bf 154}(2018), 1441--1472.

\bibitem[Jo]{Jones:survey}
Rafe Jones.
\newblock {Galois representations from pre-image trees: an arboreal survey}.
\newblock In { Actes de la {C}onf\'{e}rence ``{T}h\'{e}orie des {N}ombres et
  {A}pplications''}, volume 2013 of { Publ. Math. Besan\c{c}on Alg\`ebre
  Th\'{e}orie Nr.}, pages 107--136. Presses Univ. Franche-Comt\'{e},
  Besan\c{c}on, 2013.

\bibitem[Lu]{Luo:inhomogeneity}
Yusheng Luo.
\newblock {On the inhomogeneity of the {M}andelbrot set}.
\newblock \textit{Int. Math. Res. Not. IMRN} (2021), 6051--6076.

\bibitem[MSW]{Mavraki:Schmidt:Wilms}
Niki~Myrto Mavraki, Harry Schmidt, and Robert Wilms.
\newblock {Height coincidences in products of the projective line}.
\newblock \textit{Math. Z.}, {\bf 304} (2023), no. 2 Paper No. 26, 9pp.

\bibitem[Pi]{Pila:AO}
J.~Pila.
\newblock {O-minimality and the {A}ndr\'e-{O}ort conjecture for {$\Bbb C^n$}}.
\newblock \textit{Ann. of Math. (2)} {\bf 173}(2011), 1779--1840.

\bibitem[Sc]{Scanlon:automatic}
Thomas Scanlon.
\newblock {Automatic uniformity}.
\newblock \textit{Int. Math. Res. Not.} (2004), 3317--3326.

\bibitem[Si]{Silverman:moduli}
Joseph~H. Silverman.
\newblock \textit{Moduli spaces and arithmetic dynamics}, volume~30 of { CRM
  Monograph Series}.
\newblock American Mathematical Society, Providence, RI, 2012.

\bibitem[Yu]{Yuan:equidistribution}
Xinyi Yuan.
\newblock {Big line bundles over arithmetic varieties}.
\newblock \textit{Invent. Math.} {\bf 173}(2008), 603--649.

\bibitem[YZ]{Yuan:Zhang:quasiprojective}
Xinyi Yuan and Shouwu Zhang.
\newblock {Adelic line bundles over quasiprojective varieties}.
\newblock { Preprint, {arXiv:2105.13587v5 [math.NT]}}.

\bibitem[Zh]{Zhang:distributions}
Shou-Wu Zhang.
\newblock {Distributions in algebraic dynamics}.
\newblock In {Surveys in differential geometry. {V}ol. {X}}, volume~10 of
  { Surv. Differ. Geom.}, pages 381--430. Int. Press, Somerville, MA, 2006.

\end{thebibliography}

\def\cprime{$'$}

\bigskip\bigskip

\end{document}